\documentclass[12pt]{article}
\usepackage{amsmath,amsfonts,amssymb,amsthm,amscd}
\title{Topological dynamics and definable groups}
\date{January 11, 2012}
\author{Anand Pillay\thanks{Supported
by EPSRC grant EP/I002294/1}\\University of Leeds}
\newtheorem{Theorem}{Theorem}[section]
\newtheorem{Proposition}[Theorem]{Proposition}
\newtheorem{Definition}[Theorem]{Definition} 
\newtheorem{Remark}[Theorem]{Remark}
\newtheorem{Lemma}[Theorem]{Lemma}

\newtheorem{Fact}[Theorem]{Fact}

\begin{document}
\maketitle

\begin{abstract} 
We give a commentary on Newelski's suggestion or conjecture \cite{Newelski1} that topological dynamics, in the sense of Ellis \cite{Ellis}, applied to the action of a definable group $G(M)$ on its type space $S_{G}(M)$, can explain, account for, or give rise to, the quotient 
$G/G^{00}$, at least for suitable groups in $NIP$ theories. We give a positive answer for measure-stable (or $fsg$) groups in $NIP$ theories. As part of our analysis we show the existence of ``externally definable" generics of $G(M)$ for measure-stable groups. 
We also point out that for $G$ definably amenable (in a $NIP$ theory) $G/G^{00}$ can be recovered, via the Ellis theory, from a natural Ellis semigroup structure on the space of global $f$-generic types. 
\end{abstract}

\section{Introduction and preliminaries}
This paper concerns the relationship between two ``theories" or ``bits of mathematics". On the one hand that of a  group $G$ and its actions, by homeomorphisms, on compact spaces, i.e. abstract topological dynamics. On the other hand, that of the existence and properties of a certain canonical quotient $G/G^{00}$ for $G$ a group definable in a saturated model of a (suitable) first order theory $T$. \\

This relationship has been explored in a series of papers by Newelski, including \cite{Newelski1} and \cite{Newelski2} which are most relevant to the considerations of this paper. For {\em stable groups}, namely groups definable in stable theories,   
there is a good match, which we will briefly recall below, and the issue is whether this  extends to more general contexts. \\

A subtext of this paper as well as of Newelski's work is whether there exists a reasonably robust theory of {\em definable topological dynamics}, namely of actions of {\em definable} groups  on compact spaces.  
For example in the same way as amenability of a (discrete) group $G$ is equivalent to the existence of a $G$-invariant Borel probability measure on the compact space $\beta G$ (under the natural action of $G$), {\em definable} amenability of a group $G$ (definable in some theory $T$), as defined in \cite{NIPI} for example, is equivalent to the existence of a $G(M)$-invariant Borel probability measure on the type space $S_{G}(M)$, for $M$ some (any) model of $T$. It might then be natural to call a definable group {\em definably extremely amenable} if if for a saturated model $M$ of $T$ the action of $G(M)$ on $S_{G}(M)$ has a fixed point. And it would be also natural to ask (by analogy) whether this is equivalent to the action of $G(M)$ on $S_{G}(M)$ having a fixed point, for {\em some} model $M$ of $T$. When $T$ is stable this is indeed the case, and is equivalent to $G$ being {\em connected}. On the other hand, the Ellis theory suggests that it might be better to consider the  space $S_{G,ext}(M)$ of external types (with the natural action of $G(M)$) as being the {\em definable} analogue of $\beta G$.  
Exploration of these issues will be left to subsequent work.\\

Let us now give a brief description of the problem as posed by Newelski and of our main results, where definitions will be given later.  To begin with let $T$ be a complete first order theory with $NIP$ say,  and $G$ a $\emptyset$-definable group. Let $M$ be {\em any} model of $T$, not necessarily saturated, and $X = S_{G}(M)$ the Stone space of complete types over $M$ concentrating on $G$. So $G(M)$ acts on $X$ on the left say, by homeomorphisms. Let $(E(X),\cdot)$ be the enveloping Ellis semigroup of $X$, $I$ a minimal left ideal of $E(X)$ and $u\in I$ an idempotent. The Ellis theory yields that $u\cdot I$ is a group and the question is whether this group coincides with $G/G^{00}$. We will give a positive answer (Theorem 3.8) when the group $G$ is {\em measure-stable}. In all cases $E(X)$ coincides with $S_{G,ext}(M)$, the Stone space of the Boolean algebra of ``externally definable" subsets of $G(M)$, and in the measure-stable case, we will also show the existence of ``generics" of $S_{G,ext}(M)$ and in fact point out a one-one correspondence between these external generics over $M$ and global generic types of $G$ (Theorem 3.4). In the special case where $G$ is a definably compact group in an $o$-minimal theory, these results were obtained by Newelski 
\cite{Newelski2}. We also discuss briefly in section 2 a natural Ellis semigroup structure on the space of global $f$-generic types for $G$ a definably amenable group in an $NIP$ theory and point out that $G/G^{00}$ coincides with $u\cdot I$ ($I$ a closed left ideal and $u\in I$ an idempotent). See Proposition 2.5.
\\

In the rest of this introduction we  describe key aspects of the model-theoretic and topological dynamics contexts, as well as their interaction. We will be repeating some observations from \cite{Newelski1}, \cite{Newelski2}, but hopefully this will help to popularize the nice ideas. \\

$T$ will denote a complete first order theory in a language $L$ which for simplicity will be assumed to be countable. $x,y,..$ will usually range over finite tuples of variables.  
${\bar M}$ will usually denote a saturated model of  $T$ (say $\kappa$-saturated of cardinality $\kappa$ where $\kappa$ is inaccessible). $G$ will usually denote a $\emptyset$-definable group, often identified with its points $G({\bar M})$ in ${\bar M}$. However sometimes we pass to a larger saturated model ${\bar M}'$ in which types over ${\bar M}$ can be realized.
In general ``definability" means with parameters unless stated otherwise. 
For a model $M$, $S_{G}(M)$ denotes the set (space) of complete types $p(x)$ over $M$ which contain the formula $\phi(x)$ say which defines $G$. Identifying $G(M)$ with the collection of ``realized types" in $S_{G}(M)$, we see that $G(M)$ is a dense subset of $S_{G}(M)$.

As usual we often identify a formula with the set in defines in ${\bar M}$.
\\

We first recall $G^{00}$. 
 Let $A$ be a ``small" set of parameters from ${\bar M}$. Then there is a smallest type-definable over $A$ subgroup of $G$ which has index $< \kappa$ (equivalently index at most 
$2^{|A| + \omega}$). We call this group $G_{A}^{00}$. The quotient map $G\to G/G_{A}^{00}$ factors through the type space $S_{G}(M)$ for some (any) small model $M$ containing $A$ and equips $G/G_{A}^{00}$ with the structure of a compact (Hausdorff) topological group. When $T$ has $NIP$ (see below), $G_{A}^{00}$
does not depend on $A$ so coincides with $G_{\emptyset}^{00}$ and we simply call it $G^{00}$. So the compact group $G/G^{00}$ is a basic invariant of the definable group $G$. \\

Fix a model $M$. By an {\em externally definable} subset of $G(M)$ we mean a set of form $X\cap G(M)$ where $X$ is a definable subset of $G$ (defined with parameters possibly outside $M$). The collection of externally definable subsets of $G(M)$ is a Boolean algebra and we denote its Stone space by $S_{G,ext}(M)$, the space of ``external types" over $M$ concentrating on $G$. Let $S_{G,M}({\bar M})$ denote the (closed) subset of $S_{G}({\bar M})$ consisting of types which are finitely satisfiable in $M$.

\begin{Fact}(i) Let $p(x)\in S_{G,ext}(M)$. Then the collection of definable subsets $X$ of $G$ such that $X\cap G(M) \in p(x)$ is a complete type in $S_{G,M}({\bar M})$ which we call $p^{\bar M}$.
\newline
(ii) The map taking $p$ to $p^{\bar M}$ establishes a homeomorphism between $S_{G,ext}(M)$ and $S_{G,M}({\bar M})$. 

\end{Fact}

%Next we recall that a definable subset $X$ of $G$ is said to be left generic if finitely many left translates of $X$ cover $G$. Likewise for right generic. When left and right generics coincide we will just say generic. And $p(x)\in S_{G}(M)$ is said to be left generic if every formula in $p$ defines a left generic subset of $G$. \\

By a {\em Keisler measure} $\mu$ on $G$ over $M$ we mean a finitely additive probability measure on the collection of subsets of $G$ defined over $M$, or equivalently on the collection of definable subsets of $G(M)$. When $M = {\bar M}$ we speak of a global Keisler measure. $G$ is said to be {\em definably amenable} if it has a global left-invariant Keisler measure. From section 5 of \cite{NIPI} $G$ is definably amenable if and only if for some model $M$ there is a $G(M)$-invariant Keisler measure on $G$ over $M$.\\

A definable subset $X$ of $G$ is said to be left generic if finitely many left translates of by elements of $G$ cover $G$. Likewise for right generic. A type $p(x)\in S_{G}(M)$ is said to (left, right) generic if every formula in $p$ is.\\

In the body of this paper we will consider suitable groups $G$ in an $NIP$ theory $T$.  $T$ is said to be (or have) $NIP$ if for any formula $\phi(x,y)$, indiscernible sequence $(a_{i}:i<\omega)$, and $b$ the truth value of $\phi(a_{i},b)$ stabilizes as $i \to \infty$. If $T$ has $NIP$ then for any definable group $G$, $G^{00}$ exists (i.e. does not depend on the choice of a parameter set $A$). See \cite{NIPI} and \cite{NIPII} for background on $NIP$ theories.  A very special case of an $NIP$ theory is a stable theory, and by a {\em stable group} one just means a group definable in a stable theory. A characteristic property of a stable theory $T$ is that {\em externally definable sets are definable}. 
Much of the work on definable groups in $NIP$ theories attempts to generalize aspects of the stable case. See Chapter 1 of 
\cite{Pillay-book} for an exposition of stable group theory. In a stable group, left generic coincides with right generic (we just say generic) and generic types exist. Also $G^{00}$ coincides with $G^{0}$, the intersection of all $\emptyset$-definable subgroups of finite index, whereby $G/G^{00}$ is a profinite group. Moreover what one might call the ``fundamental theorem of stable group theory" is:
\begin{Fact} ($T$ stable.) Let $M$ be any model. Then the set $S_{G,gen}(M)$ of generic types, a closed subset of $S_{G}(M)$, is homeomorphic to $G/G^{0}$.
\end{Fact}

In section 2 we will consider definably amenable groups in $NIP$ theories, and in section 3, what we will call {\em measure-stable} groups in $NIP$ theories. The latter also go under the name of $fsg$ groups or groups generically stable for measure. They are now seen to be the right generalization of {\em stable group} in the $NIP$ setting. Definitions will be given in section 3. \\

We finish these model-theoretic preliminaries with a discussion of ``invariant" types (and forking). Suppose $M_{0}$ is a (small) model, $M>M_{0}$ is saturated with respect to $M_{0}$ (e.g. $M = {\bar M}$), and $p(x)\in S(M)$. We say that $p$ is {\em $M_{0}$-invariant} if for any $L$-formula $\phi(x,y)$ and $b\in M$, whether or not $\phi(x,b)\in p(x)$ depends only on $tp(b/M_{0})$. If $N>M$ is a bigger model, we can then define a canonical extension $p|N \in S(N)$ of $p(x)$, by defining for $\phi(x,y)\in L$ and $b\in N$, $\phi(x,b)$ to be in $p|N$ if and only if for some (any) $b'\in M$ realizing $tp(b/M_{0})$, $\phi(x,b')\in p$. 
An important example of an $M_{0}$-invariant type is $p(x)\in S(M)$ which is {\em finitely satisfiable} in $M_{0}$. \\

When $T$ is $NIP$, $p(x)\in S(M)$ is $M_{0}$-invariant if and only if $p(x)$ {\em does not fork over $M_{0}$}. The latter means that whenever $\phi(x,b)\in p(x)$ and $(b_{0}, b_{1}, b_{2},....)$ is an indiscernible over $M_{0}$-sequence with $b_{0} = b$ then $\{\phi(x,b_{i}):i<\omega\}$
is consistent. In a stable theory $T$ these notions give rise to a notion of {\em independence} with good properties. For example, $tp(a/M_{0},b)$ does not fork over $M_{0}$ if and only if $tp(b/M_{0},a)$ does not fork over $M_{0}$ and we say $a$ and $b$ are independent over $M_{0}$.\\

%We assume familiarity with the notion of a stable theory $T$ (for no $L$-formula $\phi(x,y)$ do we have $a_{i},b_{i}\in {\bar M}$ for $i\in \N$ such that $\models \phi(a_{i},b_{j})$ iff $i< j$). Stability of $T$ is equivalent to every complete type $p(x)$ over any model $M$ being definable (...) and is also equivalent to every type $p(x)$ over a saturated model being finitely satisfiable in some countable elementary submodel. In any case we have
%Mention Keisler measures and definable amenability...

%\begin{Fact} ($T$ stable.) (i) Any $p(x)\in S(M)$ has a unique extension to $p'(x)\in S({\bar M})$ which is finitely satisfiable in $M$, and moreover $p$ coincides with the global type obtained by applying the definition of $p$ to ${\bar M}$. 
%\newline
%(ii) Fix a countable submodel $M_{0}$. Then $p(x) \in S_{G}({\bar M})$ is generic if and only if for all $g\in G$, $gp$ is finitely satisfiable in and definable over $M_{0}$.
%\end{Fact} 

%$T$ is said to have $NIP$ (not the independence property) if for any formula $\phi(x,y)\in L$, infinite indiscernible sequence $(a_{i}:i<\omega)$ and $b$, the truth value of $\phi(a_{i},b)$ is eventually constant. Stability implies $NIP$. Mention invariant types etc.....

%We isolated in ..., a class of groups definable in $NIP$ theories which is a rather natural (strict) generalization of the class of stable groups. These groups go under various names, $fsg$, generically stable for measure, and as we will say here {\em measure stable}. ...\\

We now pass to topological dynamics. Our references are \cite{Ellis} as well as \cite{AAG}. 
\begin{Definition} (i) By an {\em Ellis semigroup} we mean a semigroup $(S,\cdot)$ which is a compact (Hausdorff) topological space such that $\cdot$ is continuous in the first coordinate, namely for each $b\in S$ the map taking $x$ to $x\cdot b$ is continuous.
\newline
(ii) By a closed left ideal of such an Ellis semigroup we mean a nonempty closed subset $I$ of $S$ such that $a\cdot I\subseteq I$ for all $a\in S$.
\end{Definition}

Note that by the continuity assumptions any minimal left ideal of an Ellis semigroup $S$ is closed, and moreover  such things exist.

\begin{Fact} Let $(S,\cdot)$ be an Ellis semigroup. Let $J$ be the set of idempotents of $S$ (i.e. $a\in S$ such that $a\cdot a = a$). Then 
\newline
(i) for any closed left ideal $I$ of $S$, $I\cap J$ is nonempty. 
\newline
(ii) If $I$ is minimal and $u\in I\cap J$ then $(u\cdot I,\cdot)$ is a group. 
\newline
(iii) Moreover, as $I$, $u$ vary in (ii), the groups $u\cdot I$ are isomorphic. 
\end{Fact}

Following Newelski we may call the group $u\cdot I$ above the ``ideal group" of $S$.\\

We now consider a ``$G$-flow" $(X,G)$, namely a group $G$ and a (left) action of $G$ on a compact space $X$ by homeomomorphisms. For $g\in G$, let $\pi_{g}:X\to X$ be the corresponding homeomorphism of $X$. By a subflow of $(G,X)$ we mean some $(G,Y)$ where $Y$ is a nonempty closed subspace of $X$, closed under the action of $G$ (so $(G,Y)$ is itself a $G$-flow). 

\begin{Fact} Given a $G$-flow  $(X,G)$ let $E(X)$ be the closure of $\{\pi_{g}:g\in G\}$ in (the compact space) $X^{X}$. Then
\newline
(i) $(E(X),\cdot)$ is an Ellis semigroup, where $\cdot$ is composition, and is called the {\em enveloping Ellis semigroup} of $(X,G)$.
\newline
(ii) $(E(X), G)$ is also a $G$-flow, where the action  $G$ on $E(X)$ is $\pi_{g}\circ f$. 
\newline
(iii) The minimal closed left ideals of $(E(X),\cdot)$ coincide with the minimal subflows of $(E(X), G)$.
\end{Fact}

Hence from a $G$-flow $(X,G)$, by Facts 1.4 and 1.5 we obtain a unique (up to isomorphism) group  (i.e. $(u\cdot I,\cdot)$ where $I$ is a minimal left ideal of $E(X)$ and $u$ is an idempotent in $I$). \\

We now begin connecting the two points of view. Let $T$ be a complete first order theory, $G$ a $\emptyset$-definable group, and $M$ a model. Then
$(S_{G}(M),G(M))$ is a $G(M)$-flow, $G(M)$ acting on the left. It will be convenient, now and throughout the rest of the paper, to denote by $\cdot$ the group operation on $G$ as well as the action of $G(M)$ on $S_{G}(M)$. We will also use $\cdot$ to denote the semigroup operation on $E(S_{G}(M))$ but as we point out there should be no ambiguity.
As Newelski \cite{Newelski2} observes:

\begin{Fact} (i) There is a natural homeomorphism between the compact spaces $E(S_{G}(M))$ and $S_{G,M}({\bar M})$  (global types concentrating on $G$ which are finitely satisfiable in $M$), and hence by Fact 1.1 also  $S_{G,ext}(M)$. 
\newline 
(ii) Under this homeomorphism, the Ellis semigroup operation on $E(S_{G}(M))$, becomes the following operation $\cdot$  on $S_{G,M}({\bar M})$: 
Given $p,q \in S_{G,M}({\bar M})$, let $b$ realize $q$ and let $a$ realize $p|({\bar M},b)$. Then $p\cdot q = tp(a\cdot b/{\bar M})$. 
\end{Fact}

\noindent
{\em Commentary.}   Concerning (i): Let $p\in S_{G,M}({\bar M})$ and let $a$ realise $p$. Then we have a well-defined map $\pi_{p}:S_{G}(M) \to S_{G}(M)$, given by: 
let $q\in S_{G}(M)$ be realized by $b\in G({\bar M})$. Then $\pi_{p}(q) = tp(a\cdot b/{\bar M})$. The map $\pi_{p}$ is well-defined precisely because $p$ is $M$-invariant. Suppose $q_{1},..,q_{n}\in S_{G}(M)$, realized by $b_{1},..,b_{n}\in G({\bar M})$, and $\phi_{1}(x),.., \phi_{n}(x)$ are formulas over $M$ such that $\phi_{i}(a\cdot b_{i})$ for $i=1,..,n$. Then as $p$ is finitely satisfiable in $M$ there is $a'\in G(M)$, such that $\phi_{i}(a'\cdot b_{i})$ for $i=1,..,n$. This shows that $\pi_{p}$ is in the closure of $\{\pi_{g}: g\in G(M)\}$. On the other hand, by compactness, any $f:X \to X$ in the closure of $\{\pi_{g}:g\in G(M)\}$ has the form $\pi_{p}$ for some $p\in S_{G,M}({\bar M})$. It remains to see that $p\in S_{G,M}({\bar M})$ is determined uniquely by $\pi_{p}$ and this is left to the reader, as well as (ii).\\

In the following we will identify freely $E(S_{G}(M))$, $S_{G,M}({\bar M})$ and $S_{G,ext}(M)$, denoting them by $S$, and denote by $\cdot$ the Ellis semigroup structure. As remarked earlier there is a natural embedding of $G(M)$ in $S$ and the group operation on $G(M)$ is precisely the restriction of the semigroup structure on $S$. So there is no ambiguity in denoting this semigroup operation by $\cdot$.
The following is not needed for the rest of the paper but we state it just for the record:

\begin{Remark} Let $S^{*} = S\setminus G(M)$. Then $(S,G(M),S^{*})$ is a {\em classical Ellis semigroup}  in the sense of Definition 5.2 (and Chapter 6) of 
\cite{AAG}. Namely $S$ is an Ellis monoid (with identity $e$ the identity of $G(M)$), $G(M)$ is an open dense submonoid, in fact subgroup, such that the restriction of the semigroup operation $\cdot$ to $G(M)\times S$ is continuous, $S^{*}$ is a closed subset of $S$ such that $G(M)\cup S^{*} = S$, and moreover 
$S\cdot S^{*}\cdot S = S^{*}$.
\end{Remark}

We finish this introductory section by summarizing how the Ellis theory applies to stable groups.

\begin{Fact} Suppose $T$ is stable (and as above $G$ a $\emptyset$-definable group, and $M$ any model). Then
\newline
(i) $S_{G}(M)) = E(S_{G}(M))$.
\newline
(ii) The semigroup operation on $S_{G}(M)$ is: given $p,q\in S_{G}(M)$, let $a, b$ realize $p,q$ respectively such that $a$ and $b$ are independent over $M$. Then $p\cdot q = tp(a\cdot b/M)$. 
\newline
(iii) $S_{G}(M)$ has a unique minimal closed left ideal (also the unique minimal closed right ideal) $I$ and $I$ is already a subgroup of $S_{G}(M)$. 
\newline
(iv) $I$ is precisely the collection of generic types over $M$.
\newline
(iv) $I$ (with its induced topology) is a compact topological group, isomorphic to $G/G^{0}$.

\end{Fact}

\section{Definably amenable groups}
Here we give a rather soft result for definably amenable groups $G$ in $NIP$ theories. The result is that the class of global right $f$-generic types of $G$ is, under the natural operation $\cdot$, an Ellis semigroup $S$ whose corresponding ``ideal group"  (from 1.4) is precisely $G/G^{00}$ (even as a topological group).  In fact in this case $S$ has no proper closed left ideals. \\

We first recall the relevant facts from \cite{NIPII} about definably amenable groups in $NIP$ theories. 
We assume $T$ has $NIP$.  Let us fix a countable submodel $M_{0}$ of ${\bar M}$. A definable subset $X$ of $G$ (or the formula defining $X$) is said to be left $f$-generic if for all $g\in G$, $g\cdot X$ does not fork over $M_{0}$. By \cite{Chernikov-Kaplan} we can replace ``does not fork" by ``does not divide". A global type $p\in S_{G}({\bar M})$ is said to be {\em left $f$-generic} if every formula in $p$ is left $f$-generic (equivalently, by $NIP$ for all $g\in G$, $gp$ is $Aut({\bar M}/M_{0})$-invariant). Likewise for right $f$-generic. Note that $p\in S_{G}({\bar M})$ is left $f$-generic if and only if $p^{-1}$ is right $f$-generic. 
The existence of a left (right) $f$-generic type is by 5.10 and 5.11 of \cite{NIPII} equivalent to the definable amenability of $G$. 
\\

\begin{Fact} Suppose $p(x)\in S_{G}({\bar M}$ is right $f$-generic. Then 
\newline
(i) so is $p|{\bar M}'$ for any saturated ${\bar M'}$ containing ${\bar M}$, as well as $p\cdot g$ for any $g\in G$.
\newline
(ii) $G^{00}$ is the right-stabilizer of $p$, i.e. $\{g\in G:p\cdot g = p\}$.
\end{Fact}

We now assume $G$ to be definably amenable (equivalently as mentioned above  right $f$-generic types of $G$ exist). 

\begin{Lemma} Let $S$ be the set of global right $f$-generic types of $G$. For  $p,q\in S$ define $p\cdot q$ to be 
$tp(a\cdot b/{\bar M})\in S_{G}({\bar M})$, where $b$ realizes $q$ and $a$ realizes $p|{\bar M}'$ where ${\bar M}'$ is a saturated model containing ${\bar M},b$.  Then with the induced topology from $S_{G}({\bar M})$, $(S,\cdot)$ is an Ellis semigroup. 
\end{Lemma}
\begin{proof} Note first that $S$ is a closed subset of $S_{G}({\bar M})$ so is compact. 
\newline
Secondly we show that $\cdot: S\times S \to S_{G}({\bar M})$ is continuous in the first coordinate. Let $\phi(x)$ be a formula over ${\bar M}$, say over a countable model $M$ containing $M_{0}$. Let $q\in S$ (or even in $S_{G}({\bar M})$) and let $b \in G = G({\bar M})$ realize $q|M$. Then for $p\in S$, $\phi(x)\in p\cdot q$ if and only if $\phi(b\cdot x)\in p$.
\newline
Thirdly we show that $S$ is closed under $\cdot$. Let $p, q \in S$, let $b$ realize $q$ and $a$ realize $p|{\bar M}'$ as in the statement of the lemma. By Fact 2.1(i), $tp(a\cdot b/{\bar M}')$ is right $f$-generic, hence so is $p\cdot q = tp(a\cdot b/{\bar M})$.
\newline
Finally we need to know that $\cdot$ is associative. This amounts to showing that if $p,q, r\in S$, and $a,b,c$ realize $p,q,r$ respectively such that $b$ realizes $q|{\bar M}c$ and $a$ realizes $p|{\bar M},b,c$ then $a\cdot b$ realizes $(p\cdot q)|{\bar M}c$, and this is straightforward.
\end{proof}

\begin{Lemma} $S$ has no proper left ideals (closed or otherwise). 
\end{Lemma}
\begin{proof} Let $I$ be a left ideal of $S$.
\newline
{\em Claim.} $I\cap G^{00} \neq \emptyset$, namely there is $p\in I$ such that $p(x)\models x\in G^{00}$.
\newline
{\em Proof of claim.} Let $q\in I$. So $q$ determines a coset say $C$ of $G^{00}$ in $G$. Then the coset $C^{-1}$ (as an element of $G/G^{00}$) also contains a right $f$-generic type $r\in S$. Let $p = r\cdot q$. So $p\in I$, and $p(x)\models x\in G^{00}$. 

\vspace{2mm}
\noindent
Now let $q\in S$. By Fact 2.1(ii), $q\cdot p = q$, so $q\in I$.
\end{proof} 

Note than an idempotent of $S$ is precisely any element of $S\cap G^{00}$ (by Fact 2.1 for example).

\begin{Lemma} Let $p\in S$ be an idempotent. Then $p\cdot S$ meets every coset of $G^{00}$ in $G$ in exactly one element.
\end{Lemma}
\begin{proof} First note that if $q\in S\cap G^{00}$ then $p\cdot q = p$ by Fact 2.1(ii). Hence $p\cdot S$ meets $G^{00}$ in exactly one element. On 
the other hand we know that $p\cdot S$ is a subgroup of the semigroup $S$, and as $p \cdot p = p\in p\cdot S$, $p$ is its identity element. Now suppose 
that $q,r\in p\cdot S$ are in the same coset of $G^{00}$. So working in the group $p\cdot I$, $q^{-1}\cdot r$ is in $G^{00}$ so by what we have just 
seen must $= p$. But then $p\cdot q = r$, so $q = r$. 
\end{proof}

It follows from Lemma 2.4 that the ``ideal group" $p\cdot S$ is isomorphic to $G/G^{00}$, under the map taking $q\in p\cdot S$ to the unique coset of $G^{00}$ containing $q$. But in fact this is tautologically an isomorphism of topological groups where $p\cdot S$ is given the {\em quotient topology} (with respect to the map from the compact space $I$ to $p\cdot I$  taking $q$ to $p\cdot q$). This is because we know in advance that the topology on $G/G^{00}$ is precisely that by the map $S_{G}({\bar M}) \to G/G^{00}$ and in fact also by its restriction to the compact subspace $S$.  So summarizing, we have:

\begin{Proposition} Suppose $T$ has $NIP$, $G$ is definably amenable. Let $S$ be the space of global right $f$-generic types of $G$ under the operation $\cdot$ (as in Lemma 2.2). Then, $(S,\cdot)$ is an Ellis semigroup, is itself a minimal (closed) left ideal, and for some (any) idempotent $u\in S$, the group $u\cdot S$ (with the quotient topology) is homeomorphic to $G/G^{00}$. 
\end{Proposition}

\section{Measure-stable groups} 
We again assume that $T$ has $NIP$ and $G$ is a $\emptyset$-definable group. 

\begin{Fact} The following are equivalent:
\newline
(i) There is some $p(x)\in S_{G}({\bar M})$ such that for some countable model $M_{0}$ and any $g\in G$, $g\cdot p$ is finitely satisfiable in 
$M_{0}$,
\newline
(ii) There is a global left $G$-invariant Keisler measure $\mu$ concentrating on $G$ such that $\mu$ is {\em generically stable}, i.e. for some countable model $M_{0}$, $\mu$ is both definable over and finitely satisfiable in $M_{0}$.
\end{Fact}

\noindent
{\em Commentary.} We discuss the notions of generic stability in (ii). To say that $\mu$ is definable over $M_{0}$, means that for any $L$-formula $\phi(x,y)$ and closed set $I\subset [0,1]$, the set of $b\in {\bar M}$ such that $\mu(\phi(x,b))\in I$ is type-definable over $M_{0}$. To say that $\mu$ is finitely satisfiable in $M_{0}$ means that any formula over ${\bar M}$ with positive $\mu$-measure is realized by a tuple from $M_{0}$. When $\mu$ is a type $p(x)$ we get the notion of a generically stable type. A characteristic property of stable theories is that every global type is generically stable: definable over and finitely satisfiable in some countable model $M_{0}$. 
\newline
Let us also remark that in both parts (i) and (ii) above we can replace ``some countable model $M_{0}$" by ``any countable model $M_{0}$". 
\\

Groups satisfying the equivalent conditions in Fact 3.1 were first called $fsg$ (for ``finitely satisfiable generics") groups, and later groups which are 
{\em generically stable for measure}. Here we rebaptize them  as {\em measure-stable} groups. Among measure-stable groups are stable groups, as well as definably compact groups in $o$-minimal structures and certain valued fields (algebraically closed, real closed, $p$-adically closed).
The  Keisler measure $\mu$ in Fact 3.1(ii) is in fact the unique global left-invariant Keisler measure on $G$ and also the unique right-invariant Keisler measure on $G$ (Theorem 7.7 of \cite{NIPII}). \\

%We call a definable subset $X$ of $G$ left generic if finitely many left $G$-translates of $X$ cover $G$. Likewise for right-generic. A complete type $p(x)$ (over any set) concentrating on $G$ is said to be left (right)-generic if every formula (definable set) in $p$ is left (right) - generic. Let us also mention that in fact ?? in both parts (i) and (ii) we can replace ``some countable model $M_{0}$" by ``any countable model $M_{0}$". 

\begin{Fact} Assume $G$ to be measure-stable, let $\mu$ be as in 3.1(ii), and let $X$ be a definable subset of $G$. Then the following are equivalent:
\newline
(i) $\mu(X) > 0$,
\newline
(ii) $X$ is left generic,
\newline
(iii) $X$ is right generic,
\newline
(iv) every left $G$-translate of $X$ is satisfiable in $M_{0}$ (i.e. meets $G(M_{0})$),
\newline
(v) every right $G$-translate of $X$ is satisfiable in $M_{0}$. 
\end{Fact}

Note in particular that the family of non generic definable subsets of $G$ is an ideal (in the Boolean algebra of definable subsets of $G$). \\

Let us fix now a small model $M$, which may or may not be $M_{0}$. 

\begin{Definition} (i) Let $X\subseteq G(M)$ be externally definable. We call $X$ left-generic in $G(M)$ if finitely many left translates $g\cdot X$ of $X$ by elements $g\in G(M)$ cover $G(M)$. Likewise for right-generic.
\newline
(ii) An external type $p\in S_{G,ext}(M)$ is said to be left-generic if every set in $p$ is left-generic. Likewise for right-generic.
\end{Definition}

One of our main results is:
\begin{Theorem} Assume $G$ is measure stable. Then
\newline
(i) Let $X\subseteq G$ be definable (with parameters from ${\bar M}$). Then $X$ is generic in $G$ if and only if $X\cap G(M)$ is (left, right) generic in $G(M)$.
%The map taking definable $X\subseteq G$ to $X\cap G(M)$ induces a bijection between left (right)generic definable subsets of $G$ and left (right) generic externally definable subsets of $G(M)$. 
\newline
(ii) The natural map taking definable $X\subseteq G$ to $X\cap G(M)$ induces a bijection between left (right) generic types in $S_{G}({\bar M})$ and left (right) generic types in $S_{G,ext}(M)$. 
\newline
(iii) In particular left and right generic types in $S_{G,ext}(M)$ coincide and such things exist.

\end{Theorem} 

We work towards a proof of Theorem 3.4. The main point is (i). We assume now that $G$ is measure-stable. The easy ``direction" is:
\begin{Lemma} Suppose $X\subseteq G$ is definable and $X\cap G(M)$ is (left) generic. Then $X$ is (left, so also right) generic.
\end{Lemma}
\begin{proof} Let $g_{1},..,g_{n}\in G(M)$ be such that  $G(M) = g_{1}\cdot(X\cap G(M)) \cup ... \cup g_{n}\cdot(X\cap G(M))$. Let $Z = \cup_{i}g_{i}\cdot X$ (a definable subset of $G = G({\bar M})$). Hence $Z$ contains $G(M)$, whereby the complement $Z^{c}$ of $Z$ in $G$ must be nongeneric (by Fact 3.2). Hence $Z$ is generic, whereby $X$ is generic too. 
\end{proof}

Note that it follows from  Lemma 3.5 that if $p\in S_{G,ext}(M)$ is left generic, then $p^{\bar M}$ (with notation from Fact 1.1) is a global generic type. \\

For the other direction we will make use of ``generic compact domination" from \cite{NIPIII}, as well as the following result proved in \cite{HPS} (Proposition 3.2 and its proof, as well as Corollary 3.3):
\begin{Fact} ($G$ measure-stable) (i) A global type is (left, right) generic iff it is (left, right) $f$-generic.
\newline
(ii) Moreover, suppose $W\subseteq G$ is definable and nongeneric. Let $p$ be a global generic type. Let $M'>M$ be a model over which $W$ is defined. Then for some $n$, if $(g_{1},..,g_{n})$ realizes $p^{(n)}|M'$ then $\cap_{i}g_{i}\cdot W = \emptyset$ (and also $\cap_{i}W\cdot g_{i} = \emptyset$).  
\end{Fact}

\begin{Lemma} Let $X\subseteq G$ be definable and generic. Then $X\cap G(M)$ is left and right generic

\end{Lemma} 
\begin{proof} First generic compact domination  (Proposition 5.8 of \cite{NIPIII})  gives some coset $C$ of $G^{00}$ in $G$ and some nongeneric definable subset $W$ of $G$ such that $C\subseteq X\cup W$. Now $C$ is type-definable over $M$, hence by compactness there is definable subset $D$ of $G$, defined over $M$ such that $C\subset D$ and $D\subseteq X\cup W$. Note that as $C\subset D$, $D$ is generic. Note also that $D\cap W$ is nongeneric, hence $D\cap X$ is generic. Replacing $X$ by $D\cap X$ and $W$ by $D\cap W$ we may assume that $D = X\cup W$.  Let us suppose that all the data are defined over a model $M'> M$.  Let $p$ be a global generic type of $G$ such that $p(x)\models x\in G^{00}$. Let $(g_{1},..,g_{n})$ be a realization of $p^{(n)}|M'$ as in Fact 3.6 (as $W$ is nongeneric), namely $\cap_{i}g_{i}\cdot W = \emptyset$. As each $g_{i}\in G^{00}$, we have that $g_{i}\cdot C = C$ for each $i$ and hence
$\cap_{i}g_{i}\cdot D$ contains $C$, so again by compactness there is some definable over $M$ subset $D'$ of $G$ such that $C\subset D' \subseteq \cap_{i}g_{i}\cdot D$. And note that $D'$ is generic. Now $tp(g_{1},..,g_{n}/M')$ is finitely satisfiable in $M$, hence there are $h_{1},..,h_{n}\in G(M)$ such that
\newline
(i) $\cap_{i} h_{i}\cdot W = \emptyset$, and
\newline
(ii) $D' \subseteq  \cap_{i}h_{i}\cdot D$. 
\newline
As $D = X\cap W$ it follows from (i) and (ii) that  $D'\subseteq \cup_{i=1,..,n}h_{i}\cdot X$. But as $D'$ is generic and defined over $M$, finitely many left translates of $D'$ by elements of $G(M)$ cover $G$. Hence (as the $h_{i}$ are in $G(M)$) we see that $X$ is left generic in $G(M)$. The same proof gives right generic. 
\end{proof}

\vspace{2mm}
\noindent
Lemmas 3.5 and 3.7 give part (i) of Theorem 3.4. Parts (ii) and (iii) follow immediately.\\

We can now conclude the other main result:
\begin{Theorem} ($T$ has $NIP$, $G$ is an $\emptyset$-definable measure-stable group, and $M$ is any model.) Consider the $G(M)$-flow, $(S_{G}(M), G(M))$ (with $G(M)$ acting on the left), and let $(S,\cdot)$ be the enveloping Ellis semigroup. Then for any minimal closed left ideal $I$ of $S$ and idempotent $u\in I$, the group $u\cdot I$ with its quotient topology is isomorphic to the compact group $G/G^{00}$. 
\end{Theorem}
\begin{proof} The proof is just like that of Proposition 4.8 in \cite{Newelski2} making use of Theorem 3.4 above in place of Lemma 4.6 of 
\cite{Newelski2}.  But for completeness we go through some of the details.
\newline
First by Fact 1.6 we know that $(S,\cdot)$ coincides as an Ellis semigroup with the space $(S_{G,M}({\bar M}), \cdot)$ of global types concentrating on $G$ and finitely satisfiable in $M$, where $p\cdot q = tp(a\cdot b/{\bar M})$ with $b$ realizing $q$ and $a$ realizing $p|{\bar M},b)$. Also with $S_{G,ext}(M)$ with the corresponding semigroup operation. Moreover, as  a $G(M)$-flow, $S$ coincides with  $S_{G,M}({\bar M}),G(M))$ (or $S_{G,ext}(M), G(M))$) with the natural left action of $G(M)$. 

\vspace{2mm}
\noindent
{\em Claim}  $(S,\cdot)$ has a unique (closed) left ideal $I$, consisting of the global generic types (in $S_{G,M}({\bar M})$), equivalently by Theorem 3.4 the (externally) generic types in $S_{G,ext}(M)$. 
\newline
{\em Proof.} By virtue of the above identifications, and Fact 1.5 (iii) it suffices to prove that that the class $I$ say of generic types in $S_{G,ext}(M)$ is the unique minimal subflow of $(S_{G,ext}(M),G(M))$. First we know $I$ to be is closed. Now let $I'$ be any subflow of $(S_{G,ext}(M),G(M))$. Let $p\in (S_{G,ext}(M)$ be generic, and let the externally definable set $Z\subseteq G(M)$ be in $p$. As $Z$ is generic, for any $q\in I'$, some left translate $g\cdot q$ of $q$ by some $g\in G(M)$ contains $Z$. Hence as $I'$ is a subflow, $Z$ is contained in some 
$q'\in I'$. As $I'$ is closed, $p\in I'$. Hence $I\subseteq I'$ and the claim is proved.

\vspace{2mm}
\noindent
Now $G$ is definably amenable, and by Fact 3.6 (i) and relevant definitions of the semigroup operations, the Ellis semigroup $(I,\cdot)$ of global generic types, coincides with the Ellis semigroup of right $f$-generic types of $G$ considered in section 2 (see Lemma 2.2). Hence applying Proposition 2.5  completes the proof of Theorem 3.8.

\end{proof}

\end{document}